\documentclass[a4paper,12pt,twoside]{amsart}
\usepackage[english]{babel}
\usepackage{amssymb, amsmath, eucal}
\usepackage[all]{xy}
\usepackage{mathrsfs}
\usepackage[dvips]{graphics}

\newtheorem{The}{Theorem}[section]
\newtheorem{Lem}[The]{Lemma}
\newtheorem{Prop}[The]{Proposition}
\newtheorem{Cor}[The]{Corollary}

\newtheorem{Def}[The]{Definition}

\newtheorem{Ex}[The]{Example}

\newcommand{\C}{\mathbb{C}}
\newcommand{\R}{\mathbb{R}}
\newcommand{\N}{\mathbb{N}}

\newcommand{\OO}{\mathcal{O}}
\newcommand{\HH}{\mathcal{H}}
\newcommand{\U}{\underline{u}}

\begin{document}
 \title[Parabolic complex Monge-Amp\`ere equations]{Weak solution of Parabolic complex Monge-Amp\`ere equation II}

\setcounter{tocdepth}{1}

  \author{Do Hoang Son } 
\address{Institute of Mathematics \\ Vietnam Academy of Science and Technology \\18
Hoang Quoc Viet \\Hanoi \\Vietnam}
\email{hoangson.do.vn@gmail.com }
 \date{\today\\ The author is supported by Vietnam Academy of Science and Technology, under the program ``Building a research team and research trends in complex analysis", decision VAST.CTG.01/16-17.}


\begin{abstract}  
We study the equation $\dot{u}=\log\det (u_{\alpha\bar{\beta}})-Au+f(z,t)$ in  $\Omega\times (0,T)$, where $A\geq 0, T>0$ and $\Omega$ is a bounded strictly pseudoconvex domain
   in $\C^n$, with the boundary condition $u=\varphi$ and the initial condition $u=u_0$. 
    In this paper, we consider the case where $\varphi$ is smooth and $u_0$ is an arbitrary
     plurisubharmonic function in a neighbourhood of $\bar{\Omega}$ satisfying 
     $u_0|_{\partial\Omega}=\varphi(.,0)$.
\end{abstract}

\maketitle

\tableofcontents
\newpage

\section*{Introduction}
 Let $\Omega$ be a bounded smooth strictly pseudoconvex domain of $\C^n$, i.e., there exists a smooth
 strictly plurisubharmonic function $\rho$ defined on a bounded neighbourhood of $\bar{\Omega}$ such that 
 $$
 \Omega=\{\rho<0\}.
 $$
  Let $A\geq 0,T>0$. We consider the equation
 \begin{equation}\label{KRF}
 \begin{cases}
 \begin{array}{ll}
 \dot{u}=\log\det (u_{\alpha\bar{\beta}})-Au+f(z,t)\;\;\;&\mbox{on}\;\Omega\times (0,T),\\
 u=\varphi&\mbox{on}\;\partial\Omega\times [0,T),\\
 u=u_0&\mbox{on}\;\bar{\Omega}\times\{ 0\},\\
 \end{array}
 \end{cases}
 \end{equation}
 where $\dot{u}=\frac{\partial u}{\partial t}$,
  $u_{\alpha\bar{\beta}}=\frac{\partial^2 u}{\partial z_{\alpha}\partial\bar{z}_{\beta}}$,
  $u_0$ is a plurisubharmonic function in a neighbourhood of $\bar{\Omega}$
  and $\varphi, f$ are smooth in $\bar{\Omega}\times [0,T]$.
  
  If $u_0$ is a smooth strictly 
 plurisubharmonic function in $\bar{\Omega}$ and  some compatibility conditions are satisfied on
 $\partial\Omega\times\{0\}$, then   \eqref{KRF} admits a unique solution 
 $u\in C^{2;1}(\bar{\Omega}\times [0,T))$, and then $u$ is smooth outside $\partial\Omega\times \{0\}$
  (\cite{HL10}, see also Section \ref{sec HL}). In more general cases, we define
 the weak solution 
  \begin{Def}
   The function $u\in USC(\bar{\Omega}\times [0,T))$ (upper semicontinuous function) is called a 
    \emph{weak solution} of 
    \eqref{KRF} if there exist $u_m\in C^{\infty}(\bar{\Omega}\times [0,T))$ satisfying 
    \begin{equation}\label{KRF_weak}
     \begin{cases}
     \begin{array}{ll}
     u_m(.,t)\in SPSH(\Omega),\\
     \dot{u}_m=\log\det (u_m)_{\alpha\bar{\beta}}-Au_m+f(z,t)\;\;\;&\mbox{on}\;\Omega\times (0,T),\\
     u_m\searrow\varphi&\mbox{on}\;\partial\Omega\times [0,T),\\
     u_m\searrow u_0&\mbox{on}\;\bar{\Omega}\times\{ 0\},\\
     u=\lim\limits_{m\rightarrow \infty} u_m,
     \end{array}
     \end{cases}
     \end{equation}
     where $SPSH(\Omega)=\{\mbox{strictly plurisubharmonic functions on } \Omega\}$. 
     \end{Def}
     The equation \eqref{KRF} always admits a unique weak solution, and the weak solution has been described
     in case $u_0$ has zero Lelong numbers or $u\geq N\sum\log |z-a_j|+O(1)$  (\cite{Do15a,Do15b},
     see also Section \ref{sec weak}). In this article, we will consider the general case where 
     $u_0$ has positive Lelong numbers. 
     
     The corresponding problem in compact K\"ahler manifolds was considered and solved by Di Nezza anh Lu 
     \cite{DL14}. The maximal solution (as weak solution in case of domains) is smooth
     outside $D_t\times \{t\}$, where $D_t$ is an analytic set. In case of domains on $\C^n$, by condition
     $u_0= \varphi (.,0)$ on $\partial\Omega$ (and $\varphi$ is smooth), $(u_0-\sup_{\Omega}u_0-1)$
      can be approximated as in 
     Theorem \ref{Dem appr} by functions contained in Cegrell's class $\mathcal{E} (\Omega)$ and  then for
      any $\epsilon>0$, the set $\{z\in\Omega: \nu(u_0,z)\geq \epsilon\}$ contains a finite number of 
     points \cite{Ceg04}. Hence we can describe more precisely the singular set of the weak solution: 
     $\Omega\times\{t\}$ contains only finitely many singular points of the weak solution, for any $0<t<T$.
      In the domain $\Omega\times [0,T)$, 
      the set of  singular points of the weak solution is $\cup \{a_j\}\times (0,\epsilon_j)$,
       where $\nu(u_0, a_j)>0$
      and $\epsilon_j$ is a positive number
     which is bounded by constants depending on $A$ and $\nu(u_0, a_j)$.
     
   For the convenience, we denote by $\epsilon_A$ the functions
       \begin{equation}\label{epsilonA.eq}
     \begin{cases}
         \epsilon_A(x)=\frac{x}{2n} \mbox{ if } A=0,\\
         \epsilon_A(x)=\frac{1}{A}(\log (Ax+2n)-\log (2n))  \mbox{ if } A>0.
      \end{cases}
     \end{equation}
      where $x>0$ and $0<t<T$.
     
     Our main result is  following
  \begin{The}\label{main}
  Let $A\geq 0, T>0$ and $\Omega$ be a bounded smooth strictly pseudoconvex domain of $\C^n$.
      Let $\varphi, f$ be smooth functions in $\bar{\Omega}\times [0,T]$ 
       and $u_0$ be a plurisubharmonic function in a neighbourhood of $\bar{\Omega}$
      such that $u_0(z)=\varphi(z,0)$ for any $z\in\partial\Omega$. 
      Then the weak solution u of \eqref{KRF} satisfies
   \begin{itemize}
   \item[(i)] If $m>\frac{1}{T}$ and
   $\{z\in\Omega: \nu_{u_0}(z)\geq\frac{2}{m}\}=\{a_{m,1},...,a_{m,N(m)}\}$ then there exist nonegetive numbers
   $\epsilon_{m,1},...,\epsilon_{m,N(m)}\in (0,T]$ such that 
   \begin{itemize}
   \item[(a)] $\epsilon_A(\nu (u_0,a_{m,j}))\leq \epsilon_{m,j}\leq \epsilon_A(n\nu (u_0,a_{m,j}))$ for 
   any $j=1,...,N(m)$.\\
   \item[(b)]$\nu (u(.,t),a_{m,j})>0$ for any $j=1,...,N(m)$ and $t<\epsilon_{m,j}$. Moreover, if $\epsilon_{m,j}<T$ then
   $\nu (u(.,t),a_{m,j})\searrow 0$ as $t\nearrow \epsilon_{m,j}$.\\
   \item[(c)]$u\in C^{\infty}(Q_m)$, where $Q_m= \bar{\Omega}\times (\frac{1}{m},T)\setminus
   \cup_{j\leq N(m)}\{a_{m,j}\}\times (\frac{1}{m},\epsilon_{m,j}] .$\\
  \item[(d)] $\dot{u}=\log\det (u_{\alpha\bar{\beta}})-Au+f(z,t)\;\;\mbox{on}\;
  Q_m . $\\
   
   \end{itemize}
   \item[(ii)] $u=\varphi$ on $\partial\Omega\times [0,T)$ and $\lim\limits_{t\to 0}u(z,t)=u_0(z)$ for any
    $z\in\Omega$. In particular,  $u(z,t)\stackrel{L^1}{\longrightarrow}u_0(z)$
   as $t\to 0$.\\
   \end{itemize}
  \end{The}
  
  Let us first recall some preliminaries.
\section{Coherent analytic sheaves on pseudoconvex domains}
In this section, we recall some properties of coherent analytic sheaves on pseudoconvex domains in 
$\C^n$. The readers can find  more details in \cite{Hor90} (chap VI and chap VII). Corollary \ref{Noether U}
will be used to prove Proposition \ref{Prop fini sum app}.

If $\Omega$ is a domain in $\C^n$, we let $\OO (\Omega)$ denote the set of holomorphic functions on $\Omega$.
If $z\in\C^n$, we let $\OO_z(\C^n)$ or $\OO_z$ for short denote the set of equivalence classes of functions $f$ which are holomorphic in some neighbourhood of $z$, under the equivalence relation $f\sim g$ if $f=g$ in some 
neighbourhood of $z$. If $f$ is holomorphic in a neighbourhood of $z$, we write $f_z$ for the residue class of $f$
in $\OO_z$, which is called the germ of $f$ at $z$.

It follows from  Theorem 6.3.3 and Theorem 6.3.5 in \cite{Hor90} that:
\begin{The}
Let $a\in\C^n$. Then $\OO_a$ is a Noetherian ring. If $\mathcal{I}$ is a ideal of $\OO_a$, and 
$\mathcal{I}\ni f_j\stackrel{j\to\infty}{\longrightarrow} f\in\OO_a$ in sense of simple convergence, then 
$f\in\mathcal{I}$. Here  $f_j\longrightarrow f$ means that the coefficient of $(z-a)^{\alpha}$ in $f_j$ converges
to the coefficient of $(z-a)^{\alpha}$ in $f$ for every $\alpha$.
\end{The} 

By the same argument, we also have:
\begin{The}\label{the germ}
Let $\Omega$ be a neighbourhood of $0\in\C^n$. Then for any $\{f_j\}_{j=1}^{\infty}\subset\OO_(\Omega)$, there exist
$M>0$ and $0<r<d(0,\partial\Omega)$ such that the ideal $\mathcal{I}$ of $\OO (\Delta_r^n)$ which is generated by 
$\{f_j\}_{j=1}^{\infty}$ is also generated by $\{f_j\}_{j=1}^M$. Moreover, $\mathcal{I}$ is closed in the topology
of uniform convergence on compact subsets of $\Delta_r^n$.
\end{The}
\begin{Def}
Let $X$ and $\mathcal{F}$ be two topological spaces and $\pi$ be a mapping $\mathcal{F}\to X$ such that
\begin{itemize}
\item[(i)] $\pi$ maps $\mathcal{F}$ onto $X$.
\item[(ii)] $\pi$ is a local homeomorphism, that is, every point in $\mathcal{F}$ has an open neighbourhood which
is mapped homeomorphically by $\pi$ on an open set in $X$.
\end{itemize}
Then $\mathcal{F}$ is called a sheaf on $X$ and $\pi$ is called the projection on $X$. If $U$ is a subset of $X$, a
section of $\mathcal{F}$ over $U$ is a continuous map $\varphi: U\to\mathcal{F}$ such that $\pi\varphi=Id$ on $U$.
The set of all sections of $\mathcal{F}$ over $U$ is denoted by $\Gamma (U,\mathcal{F})$. If  $x\in X$,
then $\mathcal{F}_x=\pi^{-1}\{x\}$ is called the stalk of $\mathcal{F}$ at $x$. 
\end{Def}
\begin{Ex} Let $\Omega$ be a open subset of $\C^n$. 
The sheaf $\OO_{\Omega}$ (or $\OO$ for short) of germs of holomorphic functions (analytic functions)
on $\Omega\subset \C^n$ is the topological space defined by
\begin{itemize}
\item[(i)] $\OO=\cup_{z\in\Omega}\OO_z$.\\
\item[(ii)] The topology in $\OO$ is the strongest topology such that for every open subset $U\subset\Omega$ 
and for every $f\in\OO (U)$, the map 
$U\ni z\mapsto f_z\in\OO$
is continuous.
\end{itemize}
The projection on $\Omega$ is the map $\pi : \OO\to \Omega$ defined by $\pi (\OO_z)=\{z\}$. 
The set of sections of $\OO$ over $U$ is $\Gamma (U, \OO)=\OO (U)$ and the stalk of $\OO$ at $z$ is  the set 
of germs at $z$ of holomorphic functions in some neighbourhood of $z$.
\end{Ex}
\begin{Def}
Let $\Omega$ be an open subset of $\C^n$. A sheaf $\mathcal{F}$ on $\Omega$ is called an analytic sheaf if 
it is a sheaf of $\OO$-modules, i.e., $\mathcal{F}_z$ is an $\OO_z$ module for every $z\in\Omega$
and the product of a section of $\OO$ and a section of $\mathcal{F}$ is a section of $\mathcal{F}$.
\end{Def}
 \begin{Def}
 An analytic sheaf $\mathcal{F}$ on $\Omega$ is called coherent if
 \begin{itemize}
 \item[(i)]$\mathcal{F}$ is locally finitely generated, i.e., for every $z\in\Omega$ there exists a neighbourhood
 $U\subset\Omega$ and a finite number of sections $f_1,...,f_q\in\Gamma (U,\mathcal{F})$ so that $\mathcal{F}_z$
 is generated by $(f_1)_z,...,(f_q)_z$ as $\OO_z$ module for every $z\in\Omega$.
 \item[(ii)] If $U$ is an open subset of $\Omega$ and $f_1,..., f_q\in\Gamma (U,\mathcal{F})$, then the sheaf of
 relations $\mathcal{R} (f_1,...,f_q)$ is locally finitely generated.
 \end{itemize}
 \end{Def}
 \begin{The}
 Every locally finitely generated subsheaf of $\OO^p$ is coherent.
 \end{The}
 \begin{Cor}
 Let $\{f_j\}_{j=1}^{\infty}\subset\OO_(\Omega)$. Assume that $\mathcal{F}$ is the analytic sheaf generated by
 $\{f_j\}_{j=1}^{\infty}$ over $\Omega$, i.e., the germs  $(f_1)_z,...,(f_q)_z...$ generate $\mathcal{F}_z$
 for every $z\in\Omega$. Then $\mathcal{F}$ is coherent.
 \end{Cor}
 \begin{The}
 Let $\Omega$ be a Stein manifold, $K$ an $\Omega$-holomorphically convex compact subset of $\Omega$, and 
 $\mathcal{F}$ a coherent analytic sheaf on a neighbourhood of $K$. Then
 \begin{itemize}
 \item[(i)] There exist finitely many sections $f_1, ..., f_q$ of $\mathcal{F}$ over a neighbourhood of $K$ 
 which generate $\mathcal{F}$ there.
 \item[(ii)] If $f_1,..,f_q$ are sections of $\mathcal{F}$ over a neighbourhood of $K$ which generate 
 $\mathcal{F}$ there and if $f$ is an arbitraly section of $\mathcal{F}$ over a neighbourhood of $K$, then
 one can find $c_1,...,c_q$ analytic in a neighbourhood of $K$ so that $f=\sum\limits_{j=1}^qc_jf_j$ there.
 \end{itemize}
 \end{The}
 \begin{Cor}\label{Noether U}
 Let $W$ be a pseudoconvex domains in $\C^n$ and  $\Omega$ be a open subset of $W$, $\bar{\Omega}\subset W$.
  Let $\{f_j\}_{j=1}^{\infty}\subset\OO(W)$. Assume that $\mathcal{I}$ is the ideal of $\OO(\Omega)$ 
  generated by $f_1|_{\Omega},...,f_j|_{\Omega}...$ Then there exist finitely many functions
   $f_1,...,f_q$ such that $\mathcal{I}$ is generated by $f_1|_{\Omega},...,f_q|_{\Omega}$.
 \end{Cor}
\section{Demailly's approximation}
We recall Demailly's approximation theorem, which allows to approximate plurisubharmonic functions by multiples
of logarithms of holomorphic functions. The readers can find the proof of this theorem in \cite{Dem92} (or
\cite{Dem14}). 
\begin{The}\label{Dem appr}
Let $\varphi$ be a plurisubharmonic function on a bounded pseudoconvex open set $\Omega\subset\C^n$. For every $m>0$, let $\HH_{\Omega} (m\varphi)$ be the Hilbert space of holomorphic functions $f$ on $\Omega$ such that 
$\int_{\Omega}|f|^2e^{-2m\varphi}dV_{2n}<+\infty$ and let 
$\varphi_m=\frac{1}{2m}\log\sum |g_{m,l}|^2$ where $(g_{m,l})$ is an orthonormal basis of $\HH_{\Omega}(m\varphi)$.
Then there are constants $C_1, C_2>0$ independent of $m$ such that\\
\begin{itemize}
\item[(a)] $\varphi(z)-\frac{C_1}{m}\leq \varphi_m (z)\leq \sup\limits_{|\zeta-z|<r}\varphi (\zeta) 
+\frac{1}{m}\log\frac{C_2}{r^n}\;$ for every $z\in\Omega$ and $r<d(z,\partial\Omega)$. In particular, $\varphi_m$
converges to $\varphi$ pointwise and in $L^1_{loc}$ topology on $\Omega$ when $m\to\infty$.
\item[(b)] $\nu (\varphi, z)-\frac{n}{m} \leq \nu (\varphi_m, z)\leq \nu (\varphi, z)$ for every $z\in\Omega$.
\end{itemize}
\end{The}
Using Theorem \ref{Dem appr}, we will prove the following proposition, which will be use to prove the main theorem.
\begin{Prop}\label{Prop fini sum app}
Under the assumptions of Theorem \ref{Dem appr}, for any open subset $U\Subset\Omega$,
 the following properties hold
\begin{itemize}
\item[(a)]$\int\limits_{U}e^{2m(\varphi_m-\varphi)}<\infty$.\\
\item[(b)] If $z\in U$ and $\nu (\varphi, z)< \frac{1}{m}$ then $\varphi_m$
 is smooth in a neighbourhood of $z$.\\
\item[(c)] If there exist only finitely many points $a_1,...,a_l$ in
 $\{z\in \bar{U}:\nu (\varphi, z) \geq \frac{1}{m} \}$,
  then  there exist $C,N>0$ such that
 \begin{center}
 $\varphi_m\geq N\sum\limits_{j=1}^l\log |z-a_j|-C$ on $U$.
 \end{center}
\end{itemize}
\end{Prop}
\begin{proof}
 (a) There is a corresponding result in the case of compact K\"ahler manifolds, which was proved  in
 \cite[p.164]{Dem10}, \cite[p.10]{Dem14}. The same arguments can be applied for the case of domains in $\C^n$.
 
  Set $f_j:\Omega\times\Omega\to \C$, $f_j(z,w)=g_{m,j}(z)\overline{g_{m,j}(\bar{w})}$.
 We have, for any $z,w\in V$ and $l>0$,
 \begin{center}
 $\sum\limits_{j\leq l}|f_j(z,w)|\leq (\sum\limits_{j\leq l}|g_{m,j}(z)|^2)^{1/2}
 (\sum\limits_{j\leq l}|g_{m,j}(\bar{w})|^2)^{1/2}\leq e^{m(\varphi_m(z)+\varphi(\bar{w}))}.$
 \end{center}
 Then the series $\sum\limits_{j=1}^{\infty}f_j(z,w)$ converges uniformly  on compact subsets of $\Omega\times\Omega$.
 
  Let $U\Subset V\Subset \Omega$.
  We denote by $\mathcal{J}$ the ideal of $\OO (V\times V)$ generated by 
 $\{f_j\}_{j=1}^{\infty}$. It follows from Theorem \ref{the germ} 
 that $f=\sum\limits_{j=1}^{\infty}f_j\in\mathcal{J}$ .
 
  Moreover, by Corollary 
 \ref{Noether U} we can choose $M\gg 1$ such that $\mathcal{J}$ is generated by $f_1,...,f_M$. Hence,
 \begin{center}
 $f=\sum\limits_{j=1}^Ma_jf_j,$
 \end{center}
 where $a_j\in\OO(V\times V)$. Then there exists $C>0$ such that
 \begin{center}
 $e^{2m\varphi_m}=f(z,\bar{z})\leq\sum\limits_{j=1}^M|a_j(z,\bar{z})||g_{m,j}(z)|^2\leq
 C\sum\limits_{j=1}^M|g_{m,j}(z)|^2,$
 \end{center}
 for any $z\in U$. Thus
 \begin{center}
 $\int\limits_Ue^{2m(\varphi_m-\varphi)}\leq CM<\infty.$
 \end{center}
 
(b) If $\in U$ and $\nu (\varphi, z)< \frac{1}{m}$, then $2m\varphi$ is intergrable in a neighborhood of
$z$ (see \cite{Sko72}). Hence,
 there exists $g\in\HH_{\Omega}(m\varphi)$ such that $g(z)\neq 0$. Hence $\varphi_m(z)\neq -\infty$.
 Thus $\varphi_m$  is smooth in a neighbourhood of $z$.\\
 
 (c) This is a corollary of Lemma \ref{lemgz}.
\end{proof}
\begin{Lem}\label{lemgz}
Let $g_1,...,g_M\in\OO (\Delta^n)$ such that
$$\{g_1=...=g_M=0\}=\{0\}.$$
Then there exist $C,N>0$ such that, on $\Delta^n_{1/2}$, 
$$|g_1|^2+...+|g_M|^2\geq C(|z_1|^2+...+|z_n|^2)^N.$$
\end{Lem}
\begin{proof}
It follows from Hilbert's Nullstellensatz theorem (see \cite[p.19]{Huy05}) that there exist $N>0$, $1>r>0$
and holomorphic functions $f_{jk}\in\OO (\Delta^n_r)$ satisfying, on $\Delta^n_r$,
\begin{center}
$z_k^N=\sum\limits_{j=1}^m g_j.f_{ij}$,
\end{center}
for $k=1,...,n$. Then there exists $C_1>0$ such that, on $\Delta^n_{r/2}$,
$$|g_1|^2+...+|g_M|^2\geq C_1(|z_1|^2+...+|z_n|^2)^N.$$
In the other hand, 
$\inf\limits_{\Delta^n_{1/2}\setminus\Delta^n_{r/2}}(|g_1|^2+...+|g_M|^2)>0$. Then there exists $C_2>0$ such 
that, on $\Delta^n_{1/2}\setminus\Delta^n_{r/2}$,
$$|g_1|^2+...+|g_M|^2\geq C_2(|z_1|^2+...+|z_n|^2)^N.$$
Denote $C=\max\{C_1,C_2\}$, we have, on $\Delta^n_{1/2}$,
$$|g_1|^2+...+|g_M|^2\geq C(|z_1|^2+...+|z_n|^2)^N.$$
\end{proof}
Lemma \ref{lemgz} is also an immediate corollary of Lojasiewicz inequality:
\begin{The}
Let $f: U\rightarrow \R$  be a real-analytic function on an open set $U$ in $\R^n$,
 and let $Z$ be the zero locus of $f$. Assume that $Z$ is not empty.
 Then, for any compact set $K$ in $U$, there are positive constants $C$ and $\alpha$ such that, 
 for every $x\in K$
 \begin{center}
 $dist (x,Z)^{\alpha}\leq C|f(x)|.$
 \end{center}
\end{The}
We refer the reader to \cite{Loj59, Mal66, JKS92} for more details.

Now we recall some previous results on Parabolic complex Monge-Amp\`ere equation.
\section{Parabolic complex Monge-Amp\`ere equation}
\subsection{Hou-Li theorem}\label{sec HL}
 \begin{flushleft}
 The Hou-Li theorem \cite{HL10} states
 that equation \eqref{KRF} has a unique smooth solution when the conditions are good 
 enough. We state the precise problem to be studied:
 \end{flushleft}
 \begin{equation}
 \label{HLKRF}
 \begin{cases}
 \begin{array}{ll}
 \dot{u}=\log\det (u_{\alpha\bar{\beta}})+f(t,z,u)\;\;\;&\mbox{on}\;\Omega\times (0,T),\\
 u=\varphi&\mbox{on}\;\partial\Omega\times [0,T),\\
 u=u_0&\mbox{on}\;\bar{\Omega}\times\{ 0\}.\\
 \end{array}
 \end{cases} 
 \end{equation}
 
 We first need the notion of a subsolution to \eqref{HLKRF}.
 \begin{Def}
 \label{HLsubsol}
 A function $\underline{u}\in C^{\infty}(\bar{\Omega}\times [0,T))$
 is called a \emph{subsolution} of the equation \eqref{HLKRF} if and only if
 \begin{equation}
 \label{subsolu.houli}
 \begin{cases}
 \U(.,t) \mbox{is a strictly plurisubharmonic function,}\\
 \dot{\U}\leq \log\det (\U)_{\alpha\bar{\beta}}+f(t,z,\U),\\
 \U|_{\partial\Omega\times (0,T)}=\varphi|_{\partial\Omega\times (0,T)},\\
 \U(.,0)\leq u_0 .
 \end{cases} 
 \end{equation}
 
 \end{Def}
 \begin{The}\label{houli}
 Let $\Omega\subset\C^n$ be a bounded domain with smooth boundary. Let $T\in (0,\infty]$. Assume
 that
 \begin{itemize}
 \item $\varphi$ is a smooth function in $\bar{\Omega}\times [0,T)$.
 \item $f$ is a smooth function in $[0,T)\times\bar{\Omega}\times\R$ non increasing in the
 lastest variable.
 \item $u_0$ is a smooth strictly plurisubharmonic funtion in a neighbourhood of $\Omega$.
 \item $u_0(z)=\varphi (z,0),\;\forall z\in\partial\Omega$.
 \item The compatibility condition is satisfied, i.e.
 $$
 \dot{\varphi}=\log\det (u_0)_{\alpha\bar{\beta}}+f(t,z,u_0),
 \;\;\forall (z,t)\in \partial\Omega\times \{ 0 \}.
 $$
 \item There exists a subsolution to the equation \eqref{HLKRF}.
 \end{itemize} 
 Then there exists a unique solution 
 $u\in C^{\infty}(\Omega\times (0,T))\cap C^{2;1}(\bar{\Omega}\times [0,T))$ of the equation
 \eqref{HLKRF}.
  \end{The}
 If $\Omega$ is a bounded smooth strictly pseudoconvex domain of $\C^n$ then 
 a subsolution always exists on $\bar{\Omega}\times [0,T')$ 
 (for example, $\U= M\rho+\varphi$, where $M\gg 1$ and
 $\rho$ is a smooth strictly plurisubhamonic function such that $\rho|_{\partial\Omega}=0$ ), for any 
 $0<T'<T$, and so Theorem \ref{houli} does not need the additional assumption of existence of a subsolution.
 
  Using regularity theories (see, for example, \cite{GT83, CC95, Lieb96, Do15a}), 
  we can conclude that the solution $u$
 of \eqref{HLKRF} is smooth outside  $\partial\Omega\times\{0\}$
  if the assumption of Theorem \ref{houli} holds. If $\Omega$ is a bounded smooth strictly pseudoconvex domain of $\C^n$,
   $\varphi$ is smooth in $\bar{\Omega}\times [0,T]$ and
  $f$ is smooth in $[0,T]\times\bar{\Omega}\times\R$, then $u$ is the solution of \eqref{HLKRF} 
  in $\bar{\Omega}\times [0, T+\delta)$, where $0<\delta\ll 1$. Hence $u$ can be approximated in 
  $\bar{\Omega}\times [0,T]$ by the smooth functions $u(z,t+\frac{1}{m})$. Thus, for approximations, $u$
  is as good as smooth in $\bar{\Omega}\times [0,T]$. 
 \subsection{Maximum principle}
  \begin{flushleft}
  The following maximum principle is a basic tool to
  establish upper and lower bounds in the sequel (see \cite{BG13} and \cite{IS13} for the proof).
  \end{flushleft}
  \begin{The}\label{max prin}
  Let $\Omega$ be a bounded domain of $C^n$ and $T>0$. 
  Let $\{\omega_t\}_{0<t<T}$ be a continuous family of continuous 
    positive definite Hermitian forms on $\Omega$.  Denote by $\Delta_t$ the Laplacian with 
   respect to $\omega_t$:
   $$
   \Delta_t f=\dfrac{n\omega_t^{n-1}\wedge dd^cf}{\omega_t^n},\;\forall f\in C^{\infty}(\Omega).
   $$
     Suppose that
  $H \in C^{\infty}(\Omega \times (0,T)) \cap C(\bar{\Omega}\times [0,T))$ and satisfies\\
  \begin{center}
  $(\frac{\partial}{\partial t}-\Delta_t)H \leq 0 \:$ 
  or
  $\: \dot{H}_t \leq \log\dfrac{(\omega_t+dd^c H_t)^n}{\omega_t^n}$.
  \end{center}
  Then $\sup\limits_{\bar{\Omega}\times [0,T)} H = \sup\limits_{\partial_P(\Omega \times
  [0,T))}H$. Here we denote $\partial_P(\Omega\times (0,T))=\left(\partial\Omega\times (0,T) \right)
  \cup \left( \bar{\Omega}\times\{ 0\}\right)$.
  \end{The}
  \begin{Cor}(Comparison principle)\label{compa log} 
  Let $\Omega$ be a bounded domain of $\C^n$ and $A\geq 0, T>0$. Let
   $u,v\in C^{\infty}(\Omega\times (0,T))\cap C(\bar{\Omega}\times [0,T))$ 
  satisfy:
  \begin{itemize}
  \item $u(.,t)$ and $v(.,t)$ are strictly plurisubharmonic functions for any $t\in [0,T)$,
  \item $\dot{u}\leq \log\det (u_{\alpha\bar{\beta}})-Au +f(z,t),$
  \item $\dot{v}\geq \log\det (v_{\alpha\bar{\beta}})-Av+f(z,t),$
  \end{itemize}
  where $f\in C^{\infty}(\bar{\Omega}\times [0,T))$.\\
  Then $\sup\limits_{\Omega\times (0,T)}(u-v)\leq max\{ 0, 
  \sup\limits_{\partial_P(\Omega\times (0,T))}(u-v)\}$.
  \end{Cor}
  \begin{Cor}\label{compa lap}
   Let $\Omega$ be a bounded domain of $\C^n$ and $T>0$. We denote
  by  $L$ the operator on $C^{\infty}(\Omega\times (0,T))$ given by
  $$
  L(f)=\dfrac{\partial f}{\partial t}-\sum a_{\alpha\bar{\beta}}
  \dfrac{\partial^2f}{\partial z_{\alpha}\partial \bar{z}_{\beta}}-b.f,
  $$
  where $a_{\alpha\bar{\beta}}, b\in C(\Omega\times (0,T))$, $(a_{\alpha\bar{\beta}}(z,t))$
  are positive definite Hermitian matrices and $b(z,t)<0$.\\
  Assume that  $\phi\in C^{\infty}(\Omega\times (0,T))\cap C(\bar{\Omega}\times [0,T))$ 
  satisfies\\
  $$L(\phi)\leq  0.$$
  Then $\phi\leq max(0,\sup\limits_{\partial_P(\Omega\times (0,T))}\phi).$
  \end{Cor}
 \subsection{Weak solution of Parabolic Monge-Amp\`ere equation}\label{sec weak}
 \begin{flushleft}
We recall some properties of weak solution of \eqref{KRF}, which were proved in \cite{Do15a} and \cite{Do15b}.
 \end{flushleft}
 \begin{Prop}\label{first results}
 Let $A\geq 0,T>0$ and let $\Omega$ be a bounded smooth strictly pseudoconvex domain of $\C^n$.  Assume that $u_0$ 
 is a plurisubharmonic function in a neighbourhood of $\bar{\Omega}$
   and $\varphi, f$ are smooth in $\bar{\Omega}\times [0,T]$.
    Then there exists a unique weak solution $u$  of \eqref{KRF}. Moreover, 
   \begin{itemize}
   \item[(i)] $u|_{\partial\Omega\times [0,T)}=\varphi$ ; 
   $u(z,t)\stackrel{L^1}{\longrightarrow}u_0$ as $t\searrow 0$.
   \item[(ii)] If Lelong numbers $\nu (u,a)=0$ for every $a\in\Omega$ 
   then $u\in C^{\infty}(\bar{\Omega}\times (0,T))$ and $u$ satisfies \eqref{KRF} in $\bar{\Omega}\times (0,T)$ 
   in the classical sense.
   \item[(ii)] If there exist $l\in\N, a_j\in\Omega, N_j\geq n_j\geq 0$ such that
   $$\sum\limits_{j=1}^l n_j\log |z-a_j|+C_0\geq u_0\geq \sum\limits_{j=1}^l N_j\log |z-a_j|-C_0,$$ 
       then $u$ satisfies
       \begin{itemize}
        \item[(a)] $u\in C^{\infty}(Q)$, where $Q=(\bar{\Omega}\times (0,T))\setminus
           (\cup(\{a_j\}\times (0,\epsilon_A(N_j)])$.
           \item[(b)] $u=-\infty$ on $\cup(\{a_j\}\times [0,\min\{T,\epsilon_A(n_j)\}))$.
           \item[(c)] $\dot{u}=\log\det u_{\alpha\bar{\beta}}-Au+f(z,t)$ in $Q$.
       \end{itemize}
   \end{itemize}
 \end{Prop}
 \begin{Prop} \label{weaker condition.lem.sec weak}
  Assume that there exists $u_m\in C^{\infty}(\bar{\Omega}\times [0,T))$  satisfying
  \begin{equation}\label{KRF_weak.sec weak}
     \begin{cases}
     \begin{array}{ll}
     u_m(.,t)\in SPSH(\Omega)&\forall t\in [0,T),\\
      u_m(z,t)+2^{-m}\geq u_{m+1}(z,t)&\forall (z,t)\in \bar{\Omega}\times [0,T),\\
     \dot{u}_m=\log\det (u_m)_{\alpha\bar{\beta}}-Au_m+f(z,t)\;\;&\forall (z,t)\in\Omega\times (0,T),\\
     u_m(z,t)\longrightarrow\varphi(z,t)&\forall (z,t)\in\partial\Omega\times [0,T),\\
     u_m(z,0)\longrightarrow u_0(z)&\forall z\in\bar{\Omega}.
     \end{array}
     \end{cases}
     \end{equation}
     Then $u=\lim u_m$ is a weak solution of \eqref{KRF}.
 \end{Prop}
  \begin{Prop}\label{singular.prop.weak sec}
  If there is $a\in\Omega$ such that $\nu_{u_0}(a)>0$, then the weak solution $u$ of \eqref{KRF} satisfies $u(a,t)=-\infty$ for $t\in [0,\epsilon_A(\nu_{u_0}(a)))$.
  \end{Prop}
\section{Proof of Theorem \ref{main}}
\subsection{Some technique lemmas}
We present some lemmas, which will be used to prove the main theorem. The following two lemmas were proved in 
\cite{Do15b}.
 \begin{Lem}\label{utminusu0.lem.weak sec}
 Suppose that $\psi,g\in C^{\infty}(\bar{\Omega}\times [0,T])$. 
 Assume that $v\in C^{\infty}(\bar{\Omega}\times [0,T))$ satisfies
 \begin{equation}
    \begin{cases}
    \begin{array}{ll}
    v(.,t)\in SPSH(\bar{\Omega}),\\
    \dot{v}=\log\det (v_{\alpha\bar{\beta}})-Av+g(z,t)\;\;\;&\mbox{on}\;\Omega\times (0,T),\\
    v=\psi&\mbox{on}\;\partial\Omega\times [0,T).\\
    \end{array}
    \end{cases}
    \end{equation}
    Then $$v(z,t)-v(z,0)\geq -C(t),$$ for any $(z,t)\in \bar{\Omega}\times [0,T)$. Here $C(t)$ is defined by
    $$C(t)=\inf\limits_{1>\epsilon>0}((-n\log\epsilon+A\sup |\psi|+
    \sup |g|)t-\epsilon\inf\rho)+
    \sup\limits_{t'\in [0,t]}\sup\limits_{z\in\partial\Omega}|\psi (z,t')-\psi (z,0)|,$$
    where $\rho\in C^{\infty}(\bar{\Omega})$ such that $dd^c\rho\geq dd^c|z|^2$ and $\rho|_{\partial\Omega}=0$.
     \end{Lem}
 \begin{Lem}\label{dotu.lem}
  Assume that $u\in C^{\infty}(\bar{\Omega}\times [0,T))$ satisfies
  \begin{equation}
   \begin{cases}
   \dot{u}=\log\det (u_{\alpha\bar{\beta}})+f(z,t)\;\;\mbox{on}\;\; \Omega\times (0,T),\\
   u=\varphi\;\;\;\mbox{on}\;\; \partial\Omega\times [0,T).
   \end{cases}
   \end{equation}
   Then 
  $$\dfrac{u(z,t)-\sup u_0}{t}-B\leq \dot{u}(z,t)\leq \dfrac{u(z,t)-u_0(z)}{t}+B, \;\;\forall 
  (z,t)\in\bar{\Omega}\times [0,T),$$
  where $B= 2\sup |\dot{\varphi}|+T\sup |\dot{f}|+n$ and $u_0=u(., 0)$.
  \end{Lem}
  We will also need the following elementary observation 
\begin{Lem}\label{de Giorgi}
Let $g:\R^+\rightarrow \R^+$ be a decreasing right-continuous function.  Assume that there exist $\alpha, B>1$
such that $g$ satisfies
  \begin{center}
  $tg(t+s)\leq B (g(s))^{\alpha}\;\;\forall t,s>0$.
  \end{center}
  Then $g(s)>0$ for all $s\geq s_{\infty}$, where 
  $$s_{\infty}=\dfrac{2Bg(0)^{\alpha-1}}{1-2^{-\alpha+1}}.$$
\end{Lem}
We refer the reader to \cite{EGZ09} for the proof of Lemma \ref{de Giorgi}. Using this lemma, we prove Lemma
\ref{lemestiu}.
\begin{Lem}\label{lemestiu}
Let $\Omega$ be a bounded open subset of $\C^n$. Let $u, u_0, \psi$ be plurisubharmonic functions in $\Omega$
 such that $u\in C^{\infty}(\bar{\Omega})$  and $\psi$ is bounded near $\partial\Omega$. Assume also that 
 \begin{center}
 $(dd^cu)^n\leq Be^{a_1(u-a_2u_0)}dV$,\\
 $\int_{\Omega}e^{b(\psi-u_0)}dV\leq B,$
 \end{center}
 where $a_1,a_2,b,B>0, b>a_1a_2$. Then there exists $C>0$ depending
  only on $a_1,a_2,b,B,n, \Omega$ and $\liminf\limits_{z\to\partial\Omega} (u-a_2\psi)$ such that
 \begin{center}
 $u\geq a_2\psi -C$ on $\Omega.$
 \end{center}
\end{Lem}
 There is a corresponding result in the case of compact K\"ahler manifolds, which was proved in
\cite{DL14}. The same arguments can be applied for the case of domains in $\C^n$.
 For the reader's convenience, we recall the arguments here.
\begin{proof}
 Without loss of generality, we can assume that $\liminf\limits_{z\to\partial\Omega} (u-a_2\psi)=0$.
 For $t\in\R$, we denote
  $$U_t=\{z\in\Omega: u(z)< a_2\psi (z)-t\}.$$
   Let $M>\sup |u|/a_2$ and $\psi_M=\max\{\psi, -M\}$.
 Then, for any  $t,s>0$ and $v\in PSH(\Omega)$ such that $-1\leq v\leq 0$, we have
 \begin{flushleft}
 $\begin{array}{ll}
 t^n\int\limits_{U_{t+s}}(dd^cv)^n 
 =\int\limits_{\{u<a_2\psi_M -t-s\}}(tdd^cv)^n
  &\leq \int\limits_{\{u< a_2\psi_M +tv-s\}}(dd^c(a_2\psi_M+tv-s))^n\\[14pt]
 & \leq \int\limits_{\{u<a_2\psi_M+tv-s\}}(dd^cu)^n\\[14pt]
 &\leq \int \limits_{U_s}Be^{a_1(u-a_2u_0)}dV\\[14pt]
 &\leq \int \limits_{U_s}Be^{-s}e^{a_1a_2(\psi- u_0)}dV\\[14pt]
 &\leq B\int\limits_{U_s}e^{a_1a_2(\psi- u_0)}dV\\[14pt]
 &\leq B\left(\int\limits_{U_s}e^{b(\psi-u_0)}\right)^{\frac{a_1a_2}{b}}
 \left(\int\limits_{U_s}dV\right)^{1-\frac{a_1a_2}{b}}\\[14pt]
 &\leq B^{1+\frac{a_1a_2}{b}}\lambda (U_s)^{1-\frac{a_1a_2}{b}},
 \end{array}$
 \end{flushleft}
 where $\lambda$ is Lebesgue measure. Hence, we have
 \begin{equation}\label{eq estiu capvol}
 t^nCap(U_{t+s}, \Omega)\leq B^{1+\frac{a_1a_2}{b}}\lambda (U_s)^{1-\frac{a_1a_2}{b}}.
 \end{equation}
 Moreover, it follows from \cite{AT84,Ze01} that for any $p>0$, there exists $C_p>0$ depending only on $p,n$ and
 $\Omega$ such that
 \begin{equation}\label{eq estiu volcap}
 \lambda(U_t)\leq C_p Cap (U_t,\Omega)^p.
 \end{equation}
 Combining \eqref{eq estiu capvol} and \eqref{eq estiu volcap}, we obtain
 \begin{equation}
 t^n\lambda(U_{t+s})^{1/p}\leq B^{1+\frac{a_1a_2}{b}}C_p^{1/p}\lambda (U_s)^{1-\frac{a_1a_2}{b}}.
 \end{equation}
 Let $p=\frac{2b}{b-a_1a_2}$. Applying Lemma \ref{de Giorgi} for $g(t)=\lambda(U_{t})^{1/(pn)}$ and
 $\alpha=\frac{p(b-a_1a_2)}{b}$, there exists $s_{\infty}>0$  depending only on 
 $a_1,a_2,b,B,n, \Omega$ such that $\lambda (U_{s_{\infty}})=0$. By the plurisubharmonicity of 
 $u$ and $\psi$, we conclude that $u\geq a_2\psi-s_{\infty}$.
\end{proof}
\subsection{Proof of Theorem \ref{main}} We consider the case $A=0$ and then we use it to prove the result in the case $A>0$.

\begin{flushleft}
\it{Step 1: Construct an approximation.}
\end{flushleft}
Let us first contruct a sequence of solutions as in Proposition \ref{weaker condition.lem.sec weak} such that
\begin{itemize}
\item[(i)]$\varphi_k=\varphi$ on $\partial\Omega\times (\epsilon, T)$ 
for any $\epsilon>0$ and $k\gg 1$; 
\item[(ii)] 
$\sup\limits_{k\geq 0}\sup\limits_{z\in\partial\Omega}|\varphi_k (z,t)-\varphi_k (z,0)|\rightarrow 0$
as $t\searrow 0$.
\end{itemize}
 Using the convolution of $u_0+\frac{|z|^2}{k}$ with mollifiers, 
  we can take $u_{0,k}\in C^{\infty}
  (\bar{\Omega})\cap SPSH(\bar{\Omega})$ such that
  \begin{equation}\label{u0m.eq.proof exist.sec weak}
  u_{0,k}\searrow u_0.
    \end{equation}
  Note that $u_0|_{\partial\Omega}$ is continuous. Then 
  \begin{equation}\label{deltam.eq.proof exist.sec weak}
  \delta_k=\sup\limits_{z\in\partial\Omega}(u_{0,k}(z)-u_0(z))
  \stackrel{k\rightarrow\infty}{\longrightarrow} 0.
  \end{equation}
  
   We define $g_k\in C^{\infty}(\bar{\Omega})$ and
     $\varphi_k\in C^{\infty}(\bar{\Omega}\times [0,T))$ by\\
     $$g_k=\log\det (u_{0,k})_{\alpha\bar{\beta}}+f(z,0),$$
     $$\varphi_k=\zeta(\frac{t}{\epsilon_k}) (tg_k+u_{0,k})+(1-\zeta(\frac{t}
     {\epsilon_k}))\varphi,$$
     where $\zeta$ is a smooth function on $\R$ such that $\zeta$ is decreasing,
      $\zeta|_{(-\infty,1]}=1$ and $\zeta|_{[2,\infty)}=0$. $\epsilon_k>0$ are 
    chosen such that  the sequences $\{\epsilon_k\}$,
    $\{\epsilon_k\sup|g_k|\}$ are decreasing to $0$. 
    
    $u_{0,k}$ and $\varphi_k$ satisfy the compatibility condition.
    By Theorem \ref{houli}, there exists $u_k\in C^{\infty}(\bar{\Omega}\times [0,T))$ satisfying 
     \begin{equation}\label{KRF_m.sec weak}
    \begin{cases}
    \begin{array}{ll}
    \dot{u}_k=\log\det (u_k)_{\alpha\bar{\beta}}+f(z,t)\;\;\;&\mbox{on}\;\Omega\times (0,T),\\
    u_k=\varphi_k&\mbox{on}\;\partial\Omega\times [0,T),\\
    u_k=u_{0,k}&\mbox{on}\;\bar{\Omega}\times\{ 0\}.\\
    \end{array}
    \end{cases}
    \end{equation}
  It is easy to verify that $u_k$ satisfies the conditions in Proposition \ref{weaker condition.lem.sec weak}. Then
   $u(z,t)=\lim\limits_{k\to\infty} u_k(z,t)$ is the weak solution of \eqref{KRF}.
  Moreover,  $\varphi_k=\varphi$ on $\partial\Omega\times (\epsilon, T)$ for any $\epsilon>0$ and $k\gg 1$.
   
 
\begin{flushleft}
\it {Step 2: Smoothness of weak solution in 
$(\bar{\Omega}\setminus\{\nu (u_0,z)\geq \frac{2}{m}\})\times (\frac{1}{m},T)$.}
\end{flushleft}
Let $m>\frac{1}{T}$ and $\epsilon <\frac{1}{m}$.
  Applying Lemma \ref{utminusu0.lem.weak sec} and Lemma \ref{dotu.lem} for $u_k (z,t+\epsilon)$, $k\gg 1$,
   we have,
   for any $(z,t)\in\bar{\Omega}\times (\epsilon, T)$,
  \begin{flushleft}
  $\begin{array}{ll}
  (dd^cu_k)^n =e^{\dot{u}_k-f(z,t)}dV 
  &\leq C_1e^{\frac{u_k-u_k(z,\epsilon)}{t-\epsilon}}dV\\
  &\leq C_2e^{\frac{u_k-u_{0,k}}{t-\epsilon}}dV\\
  &\leq C_2e^{\frac{u_k-u_0}{t-\epsilon}}dV,
  \end{array}$
\end{flushleft} 
where $C_1, C_2>0$ are independent of $k$.

Assume that $W$ is an open neighbourhood of $\bar{\Omega}$ such that $W$ is bounded pseudoconvex and $u_0\in PSH(W)$.
Let $l=\frac{1+\epsilon}{2(1/m-\epsilon)}$ and let $\{g_{l,j}\}$ be an orthonormal basis of 
$\HH_W (l u_0)$.
Denote $v_l=\frac{1}{2l}\log\sum\limits_{j=1}^{\infty}|g_{l,j}|^2$. Applying Lemma \ref{lemestiu} for $a_1=\frac{1}{1/m-\epsilon}$,
$a_2=1$ and $b=2l$, we have
\begin{equation}
u_k(z,\frac{1}{m})\geq v_l(z)-C_3,
\end{equation}
where $C_3>0$ is independent of $k$. Then
\begin{equation}
u(z,\frac{1}{m})\geq v_l(z)-C_3.
\end{equation}
Hence, applying Proposition \ref{Prop fini sum app}, we have
\begin{equation}\label{main s2 eq3}
u(z,\frac{1}{m})\geq N\sum\limits_{\nu(u_0,a)\geq \frac{1}{l}}\log |z-a|+O(1),
\end{equation}
where $N>0$.
It follows from Proposition \ref{first results} that $u$ is smooth and satisfies \eqref{KRF} in the
classical sense in 
$(\bar{\Omega}\setminus\{z:\nu (u,z)\geq \frac{1}{l}\})\times (\frac{1}{m},T)$.
 
 When $\epsilon\searrow 0$, we conclude that  
 $u$ is smooth and satisfies \eqref{KRF} in the
classical sense in 
$(\bar{\Omega}\setminus\{z:\nu (u,z)\geq \frac{2}{m}\})\times (\frac{1}{m},T)$ and
$u|_{\partial\Omega\times [0,T)}=\varphi|_{\partial\Omega\times [0,T)}$.

\begin{flushleft}
\it{Step 3: The set of singular points of $u$}.
\end{flushleft}
Assume that $\{z:\nu (u_0,z)\geq \frac{2}{m}\}=\{a_{m,1},...,a_{m,N(m)}\}$. For any $j=1,...,N(m)$, we denote
\begin{equation}
\epsilon_{m,j}=\sup\{T>t\geq 0: \nu(u(.,t'),a_{m,j})>0, \forall 0\leq t\leq t'\}.
\end{equation}
We need to show that
\begin{itemize}
	\item $u$ is smooth and satisfies \eqref{KRF} in the classical sense in $Q_m:=\bar{\Omega}\times (\frac{1}{m},T)\setminus
	\cup_{j\leq N(m)}\{a_{m,j}\}\times (\frac{1}{m},\epsilon_{m,j}].$
	\item $\frac{\nu (u_0,a_{m,j})}{2n}\leq \epsilon_{m,j}\leq \frac{\nu (u_0,a_{m,j})}{2}.$
	\item $\nu (u(.,t),a_{m,j})>0$ for any $j=1,...,N(m)$ and $t<\epsilon_{m,j}.$
	\item If $\epsilon_{m,j}<T$ then $\nu (u(.,t),a_{m,j})\searrow 0$ as $t\nearrow \epsilon_{m,j}$.	\\
\end{itemize}
By Step 2, for any $T>\epsilon>0$, there is $r>0$ such that $u$ is smooth in $(B(a_{m,j},r)\setminus\{a_{m,j}\})\times (\epsilon,T)$
for any $j=1,..., N(m)$.
If there is $\epsilon\in (0,T)$ satisfying $\nu(u(.,\epsilon),a_{m,j})=0$ then
it follows from Proposition \ref{first results}
that $u\in C^{\infty}(B(a_{m,j},r)\times (\epsilon,T))$. By the definition of $\epsilon_{m,j}$, we conclude that $u$ is smooth in 
$Q_m$. Clearly, $u$ satisfies \eqref{KRF} in the classical sense in $Q_m$. \\

 By Propostion \ref{singular.prop.weak sec} and by the smoothness of $u$ in $Q_m$, we have
 $\epsilon_{m,j}\geq \frac{\nu (u_0,a_{m,j})}{2n}$.
If $\nu (u_0,a_{m,j})<2T$, we have, as in the step 2,
\begin{equation}\label{main s3 eq1}
u(a_{m,j},c)\geq v_{\frac{1+\epsilon}{2(c-\epsilon)}}+O(1) =: v_l+O(1),
\end{equation}
for any $0<c<T$ and $c>\epsilon>0$.
 
If $2c>\nu (u_0, a_{m,j})$ then for any $\epsilon\ll 1$ we have
 $$\nu (u_0, a_{m,j}) <\dfrac{1}{l}.$$ 
 It follows from Proposition 
\ref{Prop fini sum app} that $v_l$ is smooth in a neighbourhood of $a_{m,j}$.  By \eqref{main s3 eq1}, we have
$$c\geq\epsilon_{m,j}.$$ Letting $c\rightarrow\frac{\nu (u_0, a_{m,j})}{2}$, 
	we obtain $\epsilon_{m,j}\leq \frac{\nu (u_0, a_{m,j})}{2}$.  Thus $\epsilon_{m,j}$ satisfies
	\begin{center}
		$\frac{\nu (u_0,a_{m,j})}{2n}\leq \epsilon_{m,j}\leq \frac{\nu (u_0,a_{m,j})}{2}.$
	\end{center}

By definition of $\epsilon_{m,j}$, we have  $\nu (u(.,t),a_{m,j})>0$ for any $j=1,...,N(m)$ and $t<\epsilon_{m,j}$. Applying Lemma
\ref{utminusu0.lem.weak sec} for $u_k$ and letting $k\rightarrow\infty$, we conclude that $\nu (u(.,t),a_{m,j})$ is 
non increasing in $t$. If $\epsilon_{m,j}<T$ then
by Proposition \ref{singular.prop.weak sec} and by the smoothness of $u$ in $Q_m$, we have, for any 
$0<\epsilon<\epsilon_{m,j}$,
\begin{center}
	$\dfrac{\nu (u(.,\epsilon),a_{m,j})}{3n}+\epsilon\leq\epsilon_{m,j}.$
\end{center}
   Hence, $\nu (u(.,t),a_{m,j})\searrow 0$ as $t\nearrow \epsilon_{m,j}$.
\begin{flushleft}
 \it{Step 4: Continuity at zero.}
\end{flushleft}
 Applying Lemma \ref{utminusu0.lem.weak sec}, we have
 \begin{equation}\label{liminf.conti0.proofmain2}
 \liminf\limits_{t\to 0}u(z,t)\geq u_0(z),
 \end{equation}
 for any $z\in\bar{\Omega}$.
 
  Note that $u$ is the limit of a decreasing sequence of smooth functions, then $u\in USC(\bar{\Omega}\times [0,T))$. We have
 \begin{equation}\label{limsup.conti0.proofmain2}
 \limsup\limits_{t\to 0}u(z,t)\leq u_0(z),
 \end{equation}
 for any $z\in\bar{\Omega}$.
 
 Combining \eqref{liminf.conti0.proofmain2} and \eqref{limsup.conti0.proofmain2}, we obtain
 $$\lim\limits_{t\to 0}u(z,t)=u_0(z).$$
 
 Hence, by the dominated convergence theorem, $u(., t)\rightarrow u_0$ in $L^1$, as $t\to 0$.
 
 \begin{flushleft}
 \it{Step 5: The case $A>0$.}
 \end{flushleft}
 Assume that $u$ is the weak solution of \eqref{KRF} with $A>0$. 
 We set $$v(z,t)=(At+1)u(z,\frac{\log (At+1)}{A}).$$ Then we can verify that $v$ is the weak solution of 
 \begin{equation}
 \begin{cases}
 \begin{array}{ll}
 \dot{v}=\log\det (v_{\alpha\bar{\beta}})-n\log (At+1)+f(z,\frac{\log (At+1)}{A})\;\;
 &\mbox{on}\; \Omega\times (0,\frac{e^{AT}-1}{A}),\\
 v(z,0)=u_0(z)&\mbox{on}\;\Omega,\\
 v(z,t)=\varphi (z,\frac{\log (At+1)}{A})&\mbox{on}\; \partial\Omega\times [0,T).
 \end{array}
 \end{cases}
 \end{equation}
 Using the case $A=0$, we conclude the similar result for the case $A>0$.

\end{document}